\documentclass{amsart}

\usepackage[T1]{fontenc}
\usepackage[utf8]{inputenc}
\usepackage{amsmath,amsthm,amsfonts,amscd,amssymb,eucal,latexsym,mathrsfs}
\usepackage{stmaryrd}
\usepackage{enumerate}
\usepackage{hyperref}
\usepackage[all]{xy}
\usepackage{etoolbox}
\usepackage{amssymb}
\usepackage{amsfonts}
\usepackage{graphicx}
\usepackage{mathtools}

\usepackage{tikz,tikz-cd}
\usetikzlibrary{shapes.geometric}
\usetikzlibrary{arrows}
\usetikzlibrary{calc}

\usetikzlibrary{positioning}
\usepackage{float}
\usepackage{MnSymbol}
\usetikzlibrary{matrix}
\usepackage{tkz-euclide}

\newcommand{\toc}{\tableofcontents}

\theoremstyle{plain}
\newtheorem{theorem}{Theorem}[section]
\newtheorem*{theorem*}{Theorem}

\newtheorem{corollary}[theorem]{Corollary}
\newtheorem*{corollary*}{Corollary}

\newtheorem{lemma}[theorem]{Lemma}

\newtheorem{proposition}[theorem]{Proposition}

\theoremstyle{definition}
\newtheorem{remark}[theorem]{Remark}

\newtheorem{question}[theorem]{Question}

\newtheorem{definition}[theorem]{Definition}
\newtheorem*{definition*}{Definition}

\usetikzlibrary{calc,graphs}
\usepackage{xcolor}
\usetikzlibrary{arrows,decorations.pathmorphing}

\newcommand{\G}{\Gamma}

\newcommand{\s}{\mathbf{s}}
\renewcommand{\r}{\mathbf{r}}
\renewcommand{\d}{\mathrm{d}}

\newcommand{\ZI}{\mathbb{Z}}

\DeclareMathOperator{\ran}{\mathrm{ran}}

\DeclareMathOperator{\impl}{\Rightarrow}

\DeclareMathOperator{\surj}{\twoheadrightarrow}

\newcommand{\Id}{\mathrm{Id}}

\newcommand{\ts}{\textsection}

\renewcommand{\r}{{\mathbf r}}
\renewcommand{\l}{{\mathbf l}}

\newcommand{\td}{\bigtriangledown}

\title{Quasi-periodicity and almost equality classes}
\author{Mikael Pichot}

\begin{document}

\maketitle
\setcounter{tocdepth}{1}

\begin{abstract}
We discuss the almost stability theorem of Dicks and Dunwoody in the context of probability measure preserving equivalence relations.
\end{abstract}

\section{Introduction}

The almost stability theorem of Dicks and Dunwoody \cite{dd} states that any $G$-stable almost equality class of subsets in a $G$-set with finite stabilizers is the vertex set of a $G$-tree having finite edge stabilizers.    Here $G$ is a group and two subsets of a set $E$ are said to be almost equal if they differ by a finite set. Almost equality is an equivalence relation on the subsets of $E$ whose classes are called almost equality classes.

In the present paper (not for publication) we describe a research project based on a conjectural extension of this theorem to  probability measure preserving (pmp) equivalence relations. The project was put on hold,  particularly due to its length and  somewhat technical nature, but recent AI breakthroughs suggest that these obstacles could soon be overcome with the help of an LLM.

Our initial motivation for studying the almost stability theorem was the finite index treeability problem for equivalence relations. Namely, it is not known if a pmp equivalence relation that contains a finite index treeable subrelation must itself be treeable. 
We recall that an equivalence relation is  treeable if each equivalence class can be measurably identified with the vertex set of some tree (see \ts\ref{S - notation} for a more precise definition).

Here we prove the following result which is at the origin of this  project.

\begin{theorem}\label{T - finite index stable almost equality classes intro}
An equivalence relation that contains a treeable subrelation of finite index admits a quasi-periodic stable almost equality class over a quasi-periodic set.
\end{theorem}

All terminology will be explained in \ts\ref{S - notation}. The quasi-periodicity condition   refers to the existence of a fundamental domain in a given $R$-set (see \ts  \ref{S - qp spaces}). 

The  following question is at the heart of this project. 

\begin{question}[Almost stability for pmp equivalence relations]\label{Q - ast} Let $R$ be a pmp equivalence relation and $E$ be a quasi-periodic countable $R$-set. Is  
any $R$-stable almost equality class over $E$ the vertex set of an $R$-tree with quasi-periodic edge set? 
\end{question}

A positive answer to this question combined with Theorem \ref{T - finite index stable almost equality classes intro}  provides: 

\begin{corollary}\label{C - main cor}
Suppose that the almost stability theorem for pmp equivalence relations holds. Then every pmp equivalence relation admitting a finite index treeable subrelation is treeable.  
\end{corollary}

Important recent works on related subjects, in particular in the context of the Stallings structure theorem, which is one of the applications of the almost stability theorem,  include \cite{T,JS,Poulain,I}. These works will be easier to discuss in a follow up note after the notions of quasi-periodic tree sets and (non-necessarily quasi-periodic) structure trees are established.

I first learned about the finite index treeability problem during my PhD thesis. It is straightforward to verify that quasi-periodicity behaves well under finite index extensions, but of course trees on equivalence classes do not lift in an obvious way.  I only became aware of  the almost stability theorem  many years later, when working on finite index subgroups in  Stallings triangles associated with Moebius--Kantor groups \cite{Stallings}. One immediately learns in \cite{dd} that trees can be endowed  in  almost equality classes, and the dual (arrow reversing) nature of this procedure. This, combined with a general understanding of quasi-periodic spaces, leads  to Theorem \ref{T - finite index stable almost equality classes intro} and its corollary. Unfortunately, these results  do not solve the finite index problem yet; it remains to answer Question  \ref{Q - ast}, which is not as trivial as it may sound. 

The proof of the almost stability theorem in \cite{dd} is divided into two chapters. The first chapter, namely \cite[Chap. II]{dd}, establishes a general and well-known existence theorem, of generating tree sets in the sense of Dunwoody, in Boolean rings of cuts that have an upper  bound on their capacity; the proof of this theorem provided in \cite{dd}  is in fact straightforward  to adapt to equivalence relations. It is a maximality argument and indeed a standard sort of arguments for equivalence relations and von Neumann algebras, and a proof for equivalence relations along the lines of \cite{dd} can be obtained with  minimal modifications. In fact, by using more recent arguments which are due to Dunwoody, one can make the result for equivalence relations  follow in an essentially obvious way from the known statements in the  group case.   

The almost stability theorem itself is proven in \cite[Chap.\ III]{dd}. The argument is long and difficult, and the proof does not  generalize easily to equivalence relations. A first remark, for instance, is the importance in the theorem of employing both the left and the right actions. This can be circumvented for equivalence relations by adding pmp assumption, which explains our restriction to pmp equivalence relations in Question \ref{Q - ast} (this restriction is superfluous for Theorem \ref{T - finite index stable almost equality classes intro}). Other issues involving finite generation can be resolved by adding an integrability assumption on the graphs on almost equality classes that arise naturally in the proofs, but the overall argument remains intricate, and certainly very much time consuming to write.

Around that time, a talented McGill student, Jonah Saks, expressed interest in working on this research topic. He started an MSc thesis under my supervision on the almost stability theorem and its application to finite index treeability; one of the basic questions we considered was the construction of a tree having as vertex set \emph{some} invariant subspace of a---general, not necessarily quasi-periodic---almost equality class. The project was initially conceived as a long term endeavour that would extend into a PhD thesis. Unfortunately, Jonah eventually decided to not continue onto a PhD after the completion of his MSc thesis. I believe, nevertheless, that the large number of talks  that Jonah gave on the subject, in various  McGill seminars between September 2021 and April 2023, were quite influential.

I plan to  update this note after gaining some familiarity with LLMs. No use of AI was made in this project, neither in the conception nor in the composition of the present document, but exciting recent announcements \cite{ai}  seem to suggest that current LLMs (as of August 2026)  excel at exploring vast technical mathematical landscapes---which would obviously be very helpful to answer Q.\ \ref{Q - ast}.   

Theorem \ref{T - finite index stable almost equality classes intro} and Cor.\ \ref{C - main cor} are established in \ts\ref{S - treeable finite index}, after a brief discussion of quasi-periodicity in \ts\ref{S - qp spaces}.  

\toc

\section{Notation}\label{S - notation}

Let $X$ be a standard Borel space. We recall that a Borel equivalence relation on $X$ is a Borel subset of $X\times X$ which is an equivalence relation.  The probability measure preserving (pmp) condition with respect a probability measure $\mu$ on $X$ requires that
\[
\int_X |K^x|\d\mu(x)=\int_X|K_y|\d\mu(y)
\]
for every Borel subset $K\subset R$ where $K^x=\{(x,y)\in K\}$ and $K_y=\{(x,y)\in K\}$ for any $x,y\in X$. 

In what follows all equivalence relations, sets, maps, etc., are implicitly assumed to be Borel.  We view an equivalence relation $R$ as a small category (groupoid) with object set $X$ and arrow set $R$.

If $X$ and $Y$ are (standard Borel) spaces and $\pi\colon Y\to X$ is a (Borel) map, we write $\G(Y)$ for the set of sections of $\pi$, i.e., (Borel) maps $s\colon X\to Y$ such that $\pi\circ s=\Id$. If $A\subset X$ is a  (Borel) subset, we write $\G((Y))$ for the set of partial sections, i.e., all $s\colon A\to Y$ such that $\pi\circ s=\Id_A$.

Let $\Sigma$ be a ``category of spaces''. In the present paper, $\Sigma$ will refer in most situation to the category of countable graphs. More generally, we mean by this one of the following: category of  countable graphs, countable sets, countable simplicial complexes, or standard Borel spaces. Other categories are suitable for the purpose of defining what a ``quasi-periodic object'' in $\Sigma$ is, but will not be useful for the purpose of this paper.  

By countable graph, we mean a pair $Y=(V,E)$ of countable sets endowed with two maps $s,r\colon E\to V^2$. A graph morphism means the usual thing. We write $VY=V$ for the vertex set and $EY=E$ for the edge set of $Y$.  A \emph{tree} is a connected acyclic countable graph.

We recall that a subset $K$ of $R$ is called a \emph{graphing} of $R$ if the connected components of the graph with vertex set $X$ and edge set $K$ are the classes of $R$. A graphing is called a \emph{treeing} if the induced (connected) graph on every equivalence class is a tree.

We use the standard notion of a Borel functor $Y$ of $R$ into $\Sigma$ ({see e.g., \cite[\ts2]{Ramsay}, or \cite[\ts4]{Connes}}). Thus, a functor $Y$ is Borel if the total space $Y(X):=\bigsqcup_{x\in X} Y(x)$ is endowed with a standard Borel structure such that the standard projection map $\pi\colon Y(X)\surj X$, the natural bijections $\pi^{-1}(x)\simeq Y(x)$ for every $x\in X$, and the action map:
\begin{align*}
R*_X Y(X) &\to Y(X)\\
 ((x,y),v)&\mapsto Y(x,y)v
\end{align*}
are all Borel maps. We will generally identify the functor $Y$ with $Y(X)$ by abuse of notation and  in particular, call $Y$ itself (instead of $Y(X)$), an $R$-space in the category $\Sigma$. 

 If  $E$ is a countable set, we write $(E,\ZI_2)$ for the power set of $E$ (we broadly follow the notation in \cite{dd}). We often view $(E,\ZI_2)$ as the set of functions on $E$ taking values in $\ZI_2:=\ZI/2\ZI$.  If $u,v\in (E,\ZI_2)$, we write $u\td v:=\{e\in E: u(e)\neq v(e)\}$ and say that $u$ and $v$ are \emph{almost equal} if this set is finite.  Almost equality is an equivalence relation on $(E,\ZI_2)$.  If $E$ is a field of countable sets over $X$ (defined by a Borel map $E\to X$  with countable fibres), almost equality is an equivalence relation on $(E^x,\ZI_2)$,  which defines a (Borel) field of (Borel) equivalence relations over the (Borel) field $(E,\ZI_2)$. We say that a Borel subfield $V\subset (E,\ZI_2)$ sits in an almost equality class if any two elements in $V^x$ are almost equal for every $x\in X$. We say that two sections $u, v\in \G(E,\ZI_2)$ sit in a uniform almost equality class if 
\[
\|u\td v\|_\infty:=\sup_{x\in X} |u^x\td v^x| < \infty.
\]
Uniform almost equality is equivalence relation on $\G(E,\ZI_2)$.

\section{Quasi-periodic spaces}\label{S - qp spaces}

Let $X$ be a (standard Borel) space, $R$ be a (Borel) equivalence relation on $X$, and $\Sigma$ be a category of spaces. In most cases, $\Sigma$ will be the category of countable graphs.  This section, which reproduces some unpublished ideas from the author's PhD thesis, is completely standard at the technical level. The goal is to introduce a notion of quasi-periodicity for spaces, for example, for graphs. This point of view has been quite useful in our study the almost stability theorem, and we take the opportunity to explain these ideas here. We will globally follow the overview \cite{qp}. Some lemmas established below will only be used in follow up notes on almost stability.

\begin{definition}\label{D - qp space}
A \emph{quasi-periodic space} (or \emph{quasi-periodic $R$-space}) in $\Sigma$ a functor $Y$ of  $R$ into $\Sigma$ which admits a fundamental domain.  
\end{definition}

By fundamental domain, we mean a (Borel) subset of $Y(X)$ which meet every $R$-orbit precisely once. Here the orbit of an element $v\in Y(X)^y$ is the set $Rv:=\{Y(x,y)v : (x,y)\in R\}$. The set of orbits form an equivalence relation on $Y(X)$ called the \emph{orbit partition} of $Y$. It is denoted $R_Y$. 
Thus, if $\Sigma$ is the category of countable graphs,  we are given a Borel field of graphs  $Y(X)$, endowed with a Borel action of $R$, for which there exists a Borel set $Y_0\subset VY(X)\sqcup EY(X)$, called a fundamental domain, which intersects every vertex $R$-orbit and every edge $R$-orbit precisely once.   We call quasi-periodic tree  a quasi-periodic graph which is a field of trees.

We recall that an equivalence relation is said to be \emph{smooth} if it admits a fundamental domain. Quasi-periodicity corresponds, by definition, to smoothness of the orbit equivalence relation $R_Y$ on $Y(X)$. The reader may well replace, in the interest of keeping the terminology concise, the word ``quasi-periodic'' by the word ``smooth'' in all that follows, provided that rudimentary maintenance is applied to the statements. Thus, for instance, it is true that a direct limit of quasi-periodic spaces is quasi-periodic, but a direct limit of smooth equivalence relations needs not be smooth. The main difference is merely conceptual in nature, in that quasi-periodicity takes place in a fixed functor category, while smoothness has no requirement that the ambient ``concept of quasi-periodicity'' (in the sense defined below) remains fixed; at the technical level, most statements are straightforward and follow easily from the standard textbooks \cite{K, sri}, even though they may not be explicitly stated in there. In fact, Theorem \ref{T - finite index stable almost equality classes intro} is one of the first proof of quasi-periodicity which does not follow immediately from first principles (it relies on  treeability).

We let $\Sigma^R_{\mathrm{qp}}$ denote the category of quasi-periodic spaces in $\Sigma$.  Morphisms in $\Sigma^R_{\mathrm{qp}}$ are defined in the usual way viewing $\Sigma^R_{\mathrm{qp}}$ as a subcategory of the functor category $\Sigma^R$ of all $R$-spaces in $\Sigma$.

\begin{definition}  
A morphism of a quasi-periodic space $Y$ to a quasi-periodic space $Z$ is a (Borel) natural transformation $Y\impl Z$. 
\end{definition}

 Explicitly, we are given, for every $x\in X$, a morphism
\[
\eta^x\colon Y(x)\to Z(x)
\]
which is Borel in the sense that the map 
\begin{align*}
\eta:=\int_{X} \eta^x \colon Y(X)&\to Z(X)\\
Y^x\ni p&\mapsto \eta^x(p)\in Z^x
\end{align*}
is Borel, and equivariant in the sense that 
\[
\eta^{\l(r)}(Y(r)p)=Z(r)\eta^{\s(r)}(p),
\]
for every $(r,p)\in R*_XY(X)$; the latter equation corresponds to commutativity of the  diagram
\begin{center}
\begin{tikzpicture}
    \matrix (m) [
      matrix of math nodes,
      row sep=1.5cm,
      column sep=1.5cm,
    ] {
          Y(\r(r))  & Y(\l(r)) \\
 	Z(\r(r))	&  Z(\l(r))  \\
  };
    \path[->]        (m-1-1) edge node[left,above] {$Y(r)$}(m-1-2)
                     (m-1-2) edge node[right] {$\eta_{\r(r)}$} (m-2-2)
    	   (m-1-1) edge node[left] {$\eta_{\l(r)}$}(m-2-1)
    	   (m-2-1) edge node[below] {$Z(r)$}(m-2-2);
\end{tikzpicture}
\end{center}
for every $(r,e)\in R*_XY(X)$. In the context of pmp equivalence relations, negligible sets are discarded as required.

We call inclusion (resp.\ surjection) of quasi-periodic spaces a natural transformation $\eta\colon Y\impl Z$ such that $\eta^x$ is almost surely injective (resp.\ surjective).

\begin{lemma}\label{L - qp invariant subspace}
Suppose $Z$ is quasi-periodic space and $Y\subset Z$ is an invariant subspace. Then $Y$ is a quasi-periodic space. 
\end{lemma}

\begin{proof}
Let  $Z_0\subset Z$ be a fundamental domain. We claim that $Y_0:=Z_0\cap Y$ is a fundamental domain of $Y$.  Indeed, if $p\in Z$ is a point, there exists a unique $r\in R$ and a unique $q\in Z_0$ such that $p=rq$. Since $Y$ is invariant, $q\in Y_0$. Clearly, if $p,q\in Y_0$ verify $p=rq$ for some $r\in R$ then $p=q$ since $p,q\in Z_0$ and $Z_0$ is a fundamental domain.  
\end{proof}

\begin{lemma}\label{L - qp stable by union and intersection} Suppose  $Y$, $Z$ are quasi-periodic subspaces of an $R$-space $V$. Then $Y\cap Z$ and $Y\cup Z$ are quasi-periodic subspaces of $V$. 
\end{lemma}

\begin{proof} Clearly, $Y\cap Z$ is an $R$-stable subspace of $Y$, and is therefore a quasi-periodic countable set by Lemma \ref{L - qp invariant subspace}. Similarly, $Y\setminus Z$ and $Z\setminus Y$ are quasi-periodic. 
The lemma then follows by the obvious fact that a disjoint sum of quasi-periodic spaces is again quasi-periodic.
\end{proof}

\begin{lemma}\label{L - qp direct product}
Suppose $Y$ is an $R$-space and $Z$ is a quasi-periodic space. Then $Y\times_X Z$ is a quasi-periodic space. 
\end{lemma}

Here $Y\times_X Z$ refers to the fibered product (pullback) over $X$.

\begin{proof}
Let $Z_0$ be a fundamental domain for $Z$. We claim that $F=Y\times_XZ_0$ is a fundamental domain for $Y\times_X Z$. Indeed, if $(p,q)\in Y\times_X Z$, there exists a unique $r\in R$ and $q_0\in Z_0$ such that $q=rq_0$. Let $p_0=r^{-1}p\in Y$. Then $(p_0,q_0)\in F$ and $(p,q)=r(p_0,q_0)$. If   $(p,q)=r(p',q')$ and $(p,q),(p',q')\in F$, then $p=rp'$ and $q=rq'$. Since $q,q'\in Z_0$, it follows that $r=1$. Therefore, $p=p'$ and $q=q'$.  
\end{proof}

\begin{lemma}\label{L - pair qp spaces}
Suppose $Y\subset Z$ is a nested pair of quasi-periodic spaces. Then there exist a fundamental set $Y_0$ of $Y$ and a fundamental set $Z_0$ of $Z$ such that $Y_0\subset Z_0$.
\end{lemma}

\begin{proof} Let $Y_0$ be any fundamental set of $Y$ and $Z_0$ be any fundamental set of $Z$. Since $Z\setminus Y$ is an invariant Borel subspace of $Z$, it is quasi-periodic by Lemma \ref{L - qp invariant subspace}, which furthermore admits $Z_0\setminus Y$ as a fundamental domain. It follows that the set $Y_0\cup (Z_0\setminus Y)$ is a fundamental set of $Z$ which contains $Y_0$.  
\end{proof}

For our next lemma we note that the direct limit of a nested countable family of measurable $R$-bundles in $\Sigma$ is a measurable $R$-bundle in $\Sigma$. 

\begin{lemma}\label{L - qp spaces increasing unions} Countable increasing unions and decreasing intersections of quasi-periodic spaces are quasi-periodic.
\end{lemma}

\begin{proof}
Suppose $(Z^n)_{n\geq 1}$ is a countable increasing union of quasi-periodic countable sets. It follows by applying Lemma \ref{L - pair qp spaces} inductively that there exists an increasing sequence $Z_0^1\subset Z_0^2\subset \cdots$ of measurable sets such that  for every $n\geq 1$, the set $Z_0^n$ is a fundamental domain of  the space $Z^n$. The direct limit $Z=\varinjlim Z^n$ of the sequence $(Z^n)_{n\geq 1}$ is a measurable $R$-bundle, and the direct limit $Z_0=\varinjlim Z_0^n$ of of the sequence $(Z_0^n)_{n\geq 1}$ is a mesurable subset of $Z$. If $p\in Z$ then there exists an integer $n$ such that $p\in Z_n$. Thus, there exists an $r\in R$ and $p_0\in Z^n_0\subset Z_0$ such that $p=rp_0$. Similarly, $p,q\in Z_0$ verify $p=rq$ for some $r\in R$, then there exists an integer $n$ such that $p,q\in Z_n$. Thus, $r=1$ and $p=q$.  If $(Z^n)_{n\geq 1}$ is a countable decreasing intersection of quasi-periodic spaces, their intersections $Z$ is an invariant subset of $Z^1$, and in particular is quasi-periodic.
\end{proof}

This shows that the full subcategory $\Sigma^R_{\mathrm{qp}}\subset \Sigma^R$ of quasi-periodic spaces in $\Sigma$ is stable under invariant subsets, complements, and increasing countable unions and intersections. Next, we discuss lifting and quotient stability.

\begin{lemma}\label{L - qp  R  map }
Suppose $Y, Z$ are $R$-spaces, and $\pi\colon Y\to Z$ is a morphism. If $Z$ is quasi-periodic so is $Y$. 
\end{lemma}

\begin{proof}
Suppose now that $Z$ is quasi-periodic and let $Z_0$ be a fundamental domain for $Z$. Let $Y_0:=\pi^{-1}(Z_0)$. Clearly,  $Y_0$ is Borel. If $p\in Y$, then there exists a a unique $r\in R$ and a unique $q_0\in Z_0$  such that $r\pi(p)\in Z_0$. Since $\pi$ is equivariant,  $rp\in Y_0$.  Next we prove that $Y_0$ is a fundamental domain. Suppose $p,q\in Y_0$ verify $q=rp$ for some $r\in R$.  Then $\pi(q)=r\pi(p)\in Z_0$. Thus, $r=1$ and $p=q$. This prove that $Y$ is quasi-periodic.
\end{proof}

\begin{lemma}\label{L - qp  finite to 1}
Suppose $Y, Z$ are $R$-spaces, and $\pi\colon Y\to Z$ is an essentially surjective finite-to-1 morphism. Then  $Y$ is quasi-periodic if and only if $Z$ is quasi-periodic. 
\end{lemma}

\begin{proof}
Suppose first that $Y$ is quasi-periodic and let $Y_0$ be a fundamental domain for $Y$. We prove that $Z_0:=\pi(Y_0)$ contains a fundamental domain for $Z$. Clearly, $Z_0$ is Borel since $\pi$ is finite-to-1. Let $q\in Z$. Since $\pi$ is surjective there exists $p\in Y$ such that $q=\pi(p)$. Let $r\in R$ and $p_0\in Y_0$ such that $p=rp_0$. Then $q_0:=\pi(p_0)\in Z_0$ and since $\pi$ is an $R$-map, we have $rq_0=r\pi(p_0)=\pi(rp_0)=\pi(p)=q$. 

Next we prove that the restriction to  $Z_0$ of the action of $R$ on $Z$ has finite classes. Let $q,q_1,\ldots,q_n \in Z_0$ be elements such that there exists $r_1,\ldots,r_n \in R$ verifying $q=r_iq_i$. Choose $p,p_1,p_2, \ldots\in Y_0$ such that $q=\pi(p)$ and $q_i=\pi(p_i)$. Then $\pi(r_ip_i)=q$, so $r_ip_i\in \pi^{-1}(q)$. Furthermore, the map $\bigvee_{i=1}^n \{p_i\}\to \pi^{-1}(q)$ taking $p_i$ to $r_ip_i$ is injective. Indeed, if $r_ip_i=r_jq_j$, then $p_i=r_i^{-1}r_jp_i$ which implies $r_i=r_j$ since $p_i,p_j\in Y_0$. Thus, $p_i=p_j$.  In particular, $n\leq |\pi^{-1}(q)|$. This prove that the class of $q\in Z_0$ is finite. Since any Borel finite equivalence relation admits a fundamental domain, we have shown that $Z$ is quasi-periodic. 

The converse follows by the previous lemma.
\end{proof}

If we assume that $\Sigma$ is stable by quotient by finite groups of automorphisms, this implies the following result.

\begin{lemma}\label{L - qp quotient by finite group}
Suppose $Y$ is a quasi-periodic space and $H$ is a finite group acting by fibred automorphisms of the field $Y(X)$ commuting to $Y$. In other words, we assume that $h\in H$ acts as an automorphism of $Y(x)$ in the category $\Sigma$, such that $rh=hr$ for every $r\in R$. Then the space $Z:=Y/H$ of $H$-orbits is quasi-periodic. 
\end{lemma}

\begin{proof}
Here we assume that  $Z(X):=Y(X)/H$ is endowed with the quotient Borel structure, such that the  map 
\begin{align*}
\pi\colon Y(X)&\to Z(X)\\
p&\mapsto Hp
\end{align*}
is a finite-to-1 Borel map. Since the action of $R$-commutes to the action of $Y$, it is clear that $Z$ is an $R$-space and $\pi$ is an $R$-map. If $Y_0$ is a fundamental domain for $Y$, then $\pi(Y_0)$ contains a fundamental domain for $Z$ by the previous lemma.
\end{proof}

Next we discuss stability under subrelation.

\begin{lemma}\label{L -qp subrelation}
 Suppose $S$ is a relation of $R$. If $Y$ is a quasi-periodic $R$-space then $Y$ is a quasi-periodic $S$-space.
\end{lemma}

\begin{proof}
Let $R_0$ be a fundamental domain for the left action of $S$ on $R$. Let $Y_0$ be a fundamental domain of $Y$ for $R$. We claim that $Y_1:=\{ry : r\in  R_0, \ y\in Y\}$ is a fundamental domain for $S$. Namely, if $y\in Y$ then there exists an $r\in R$ and a $y_0\in Y_0$ such that $y=ry_0$. If $r_0\in R_0$ and $s\in S$ verify $r=sr_0$ then $y=s(r_0y_0)$ where $r_0y_0\in Y_1$. Furthermore, if $sr_0y_0=r_1y_1$ for some $s\in S$, $y_0,y_1\in Y_0$ and $r_0, r_1\in R_0$, then $sr_0=r_1$ since $Y_0$ is a fundamental domain for $R$, and so $s=1$ since $R_0$ is a fundamental domain for $S$.
\end{proof}

\begin{lemma} \label{L - qp finite index}
Suppose $S$ is a subrelation of $R$ of finite index and $Y$ is an $R$-space. Then $Y$ is a quasi-periodic $R$-space if and only if $Y$ is a quasi-periodic $S$-space.
\end{lemma}

\begin{proof}
Let $Y$ be an $S$-quasi-periodic $R$-space. Let $Y_0$ be a fundamental domain for $S$. Since $[R:S]<\infty$, the restriction $R_0$ of the orbit partition of the $R$-space $Y$ to $Y_0$ is an fsr. Thus, we may choose a fundamental domain $Y_1\subset Y_0$ for this equivalence relation. This is straightforward to check that $Y_1$ is a fundamental domain of $Y$ viewed as an $R$-space. Namely, if $y\in Y$, take $s\in S$ and $y_0\in Y_0$ such that $y=sy_0$. There exists a unique $y_1\in Y_1$ such that $(y_0,y_1)\in R_0$. Take $r_0\in R$ such that $y_0=ry_1$. Then $y=sr_0y_1$. Suppose $y_0,y_1\in Y_1$ verify $ry_0=y_1$. Since $Y_1\subset Y_0$, this implies that $r\in R_0$. Since $y_0,y_1\in Y_1$, this implies $r=1$. Thus, $y_0=y_1$, establishing uniqueness. 
\end{proof}

We now turn to results that are more specific to the category of graphs.

\begin{lemma}
Suppose $Y$ is an $R$-graph. If the $R$-map $EY\to VY\times_X VY$ is finite-to-1 and $VY$ is quasi-periodic, then $EY$ is quasi-periodic. 
\end{lemma}

In particular, if $Y$ is a simplicial graph (e.g., a tree), then $VY$ quasi-periodic implies $EY$ quasi-periodic. 

\begin{proof} It follows by assumption and  Lemma \ref{L - qp direct product} that $VY\times_X VY$ is quasi-periodic. 
Since $EY\to VY\times_X VY$ is an $R$-map, it follows by Lemma \ref{L - qp invariant subspace} that its image in $VY\times_X VY$ is a quasi-periodic subset. Lemma \ref{L - qp  finite to 1} now implies that $EY$ is quasi-periodic. 
\end{proof}

The converse is of course not true; however, we have the following for connected locally finite graphs.  

\begin{lemma}
Suppose $Y$ is a connected locally finite $R$-graph. Then  $VY$ is quasi-periodic if and only if $EY$ is quasi-periodic. 
\end{lemma}

\begin{proof}
If $Y$ is locally finite, then both maps $\s,\r$ are finite-to-1. It follows by Lemma \ref{L - qp  finite to 1} that $\s(EY)$ and $\r(EY)$ are invariant quasi-periodic subsets of $VY$. Since $Y$ is connected, we have $VY=\s(EY)\cup \r(EY)$. The result follows by Lemma \ref{L - qp stable by union and intersection}. 
\end{proof}

\begin{lemma}\label{L - quasi-periodic vertex set}
Suppose $Y$ is a connected locally finite $R$-graph. If $VY$ contains a quasi-periodic set, then $VY$ is quasi-periodic.   
\end{lemma}

\begin{proof}
Let  $V_0$ be a quasi-periodic subset of $VY$ and let $V_n$ denotes the set of vertices at distance at most $n$ from $V_0$ in $VT$. This set admits an equivariant finite-to-1 retraction onto $V_0$ and it follows by Lemma \ref{L - qp  finite to 1} that $V_n$ is quasi-periodic. Since $VT=\bigcup_{n\geq 1} V_n$, Lemma \ref{L - qp spaces increasing unions} implies that $VT$ is quasi-periodic. \end{proof}

We conclude with an important result, which I liked to phrase in these terms in my PhD thesis \cite{qp}, but which is of course just a reformulation of the results of Gaboriau in  \cite{Gcost,Gaboriau}.

\begin{proposition} \label{P - treeable qp} An equivalence relation is treeable if and only if it admits a quasi-periodic tree.
\end{proposition}

 In fact, in terms of quasi-periodic spaces, the definition of treeability is precisely the requirement that there exists a quasi-periodic tree whose vertex set is precisely $R$; the above lemma adds some flexibility on the possible vertex sets, as long as they are quasi-periodic. It is very convenient to allow for larger vertex sets in the context of the almost stability theorem.

\begin{proof}
The equivalence follows immediately by \cite{Gaboriau}. Namely, if $Y$ is a quasi-periodic tree, we may embed a fundamental domain $V_0Y$ of $VY$ into a fundamental domain of the vertex set of the universal $R$-simplicial complex $ER$ (defined in \cite[\ts 2.2.1]{Gaboriau}). This embedding extends to an equivariant embedding of $VY(X)$, and identifies the $R$-tree $Y(X)$ with a 1-dimensional  contractible simplicial $R$-complex in $ER$, which implies that $R$ is treeable by \cite{Gaboriau}. The converse is clear since the Cayley graph of a treeing is a quasi-periodic tree, with vertex set precisely equal to $R$.
\end{proof}

\begin{remark} The proof of the fact that $R$ is treeable if it admits 1-dimensional  contractible $R$-simplicial complex  is not written explicitly in \cite[\ts 2.2.1]{Gaboriau}; however, it follows by \cite[Prop.\ II.6]{Gcost}, since $R$ is stably isomorphic to the treeable equivalence relation on the standard Borel space $VY(X)/R$, whose classes are the vertex sets of the trees $Y^x$, $x\in X$. 
\end{remark}

We conclude with this section with a brief discussion of the ``probabilitic approach'' to quasi-periodic spaces. 

Let $Y$ be a quasi-periodic space. As discussed in \cite{qp}, we may view the functor $Y$ as a \emph{random variable}, with each fibre $Y^x$, $x\in X$, being a \emph{realization}, in the probabilistic sense of the word, of the quasi-periodic space $Y$. This is an analog of an elementary event in probability theory or a concrete realization of a random variable. 

In this approach, a \emph{point} in $Y$ is the choice of a point in every realization $Y^x$. Thus, in standard terminology, a point in $Y$ is simply a section in $\G(Y)$. If $Y$ is a quasi-periodic graph, or more generally a field of graphs, we call  \emph{vertex} in $Y$  an element of $\G(VY)$ and  \emph{edge} in $Y$ an element in $\G(EY)$. By the standard selection theorems, the vertex set of a countable quasi-periodic graph is a union of countably many vertices, which can moreover be chosen to uniformly apart from each other: 

\begin{lemma}
Suppose $Y$ is a field of countable graphs. There exists a countable set of vertices $v_0,v_1,v_2,\ldots $ of $Y$ such that 
\begin{enumerate}
\item $Y^x=\{ v_0^x, v_1^x,v_2^x,\ldots\}$ for every $x\in X$;
\item $d_\infty(v_n,v_0)<\infty$
\end{enumerate} 
where $d_\infty(v_n,v_0):=\sup_{x\in X} d^x(v_n^x,v_0^x)$ denotes the uniform graph distance between vertices. 
\end{lemma}

\begin{proof}
Let $v_0\in \G(Y)$ be a an arbitrary vertex in $Y$. Since the map $VY\to X$ is countable-to-1, the Lusin--Novikov theorem provides a countable  family of sections $w_k\colon A_k\to VY$ whose graph partition $VY$. Furthermore, we may assume by restricting $A_k$ if necessary that  $\|w_k-{v_0}_{|A_k}\|_\infty <\infty$ for all $k$. Then the family of vertices $v_k$ defined by $v_k(x)=w_k(x)$ if $x\in A_k$ and $v_k(x)=v_0(x)$ if $x\not\in A_k$ verifies (1) and (2). 
\end{proof}

If $A$ is a subset of $X$, a section $v\colon A\to VY$ in $\G((VY))$ is called a \emph{partial vertex} of $Y$. A partial vertex may or may not occur in a given concrete realization $Y^x$. If $\mu$ is an invariant probability measure on $X$, then $\mu(A)$ represents the \emph{probability of occurrence} of the partial vertex $v$ in the quasi-periodic graph $Y$.

By subset of $Y$, we mean a (Borel) subset of $Y(X)$. Let $Z$ be a subset of $Y$. We  define the stabilizer of $Z$ to be
 \[
 R_Z=\{(u,v)\in Z\times Z : \exists (x,y)\in R \  (x,y)v=u\}.
 \]
Thus, $R_Z$ is by definition the restriction to $Z$ of the orbit partition of $R$ acting on $Y$. This notation is consistent with the notation $R_Y$ introduced earlier.

If $Y$ is an $R$-space, we say a subset $Z$ of $Y$ is free (resp., smooth) if $R_Z$ is trivial (resp., smooth). The $R$-space spanned by $Z$, also called the saturated of $Z$ under the action of $R$, is denoted $RZ$. Thus, by definition, $RZ$ is the union of all $R$-orbits intersecting $Z$. This is an invariant (Borel) subset of $Y$.

If $v\colon A\to Y$ is a partial point of $Y$ (e.g., a partial vertex in an $R$-graph), then we use the notation $R_v$ for the stabilizer of $v$ (i.e., the stabilizer of $X=v(A)$), and $Rv$ for the $R$-space spanned by $v$. 
Since $v$ is injective, $R_v$ can be pulled back to a subrelation of $R$ defined on $A$, namely, $R_v=\{(x,y) : (x,y)v^y=v^x\}\subset A\times A$.  If $Y$ is quasi-periodic, then all partial points in $Y$ are smooth by definition, and all partial points in a fundamental domain of $Y$ are free. Furthermore, it is straightforward to verify that a countable $R$-set $Y$ is quasi-periodic iff it is a union of countably many smooth points:

\begin{lemma}\label{L - qp sections}
Let $V$ be a countable $R$-space, let $v_1,\ldots, v_n,\ldots \in \G(V)$, and let $V_n=Rv_n$ denote the $R$-space spanned by $v_n$. Then the $R$-space $\bigcup_{n\geq 1} V_n$ is quasi-periodic if and only if $V_n$ is quasi-periodic for every $n\geq 1$. 
\end{lemma}

\begin{proof} Suppose each $V_n$ is quasi-periodic. 
Then $W_n:=\bigcup_{m\leq n} V_n$ is quasi-periodic by Lemma \ref{L - qp stable by union and intersection}. Therefore, $\bigcup_{n\geq 1} V_n=\bigcup_{n\geq 1} W_n$ is quasi-periodic by Lemma \ref{L - qp spaces increasing unions}. The converse follows by Lemma \ref{L - qp invariant subspace}.\end{proof}

I first encountered quasi-periodic sets when studying the concentration of measure property in the sense of Gromov \cite{Gromov} (i.e., for foliations and mm-spaces) and the connections beetween this property (which were suggested by Gromov) to amenability and Kazhdan's property T. In the context of \cite{Gromov},  it was helpful, in order to establish concentration, or the lack thereof, to view a foliation of a compact Riemannian manifold (say) as a quasi-periodic Riemannian manifold $Y$. The concentration property is easier to visualize in terms of relative distance between quasi-periodic subsets of a quasi-periodic space, rather than in the ambient compact manifold or mm-space. To my surprise, no notion of quasi-periodic space seemed to exist in the mathematical literature. The theory of almost periodic functions of von Neumann, and the theory of aperiodic tilings, provided related ideas, but they were not well adapted to measure theory and concentration.  

Two different quasi-periodic spaces may be subject to the ``same'' notion quasi-periodicity (i.e., patterns repeat in both spaces in the ``same'' way). We will posit here (as was done in \cite{qp}) that the ``underlying concept of quasi-periodicity'' for a space, or in other words, the underlying structure that regulates the quasi-periodicity of a space, is a stable isomorphism classes of  equivalence relations. Let us first recall the definition of stable isomorphism. 

\begin{definition}
 A pmp equivalence relation $R$ is \emph{stably isomorphic} to a pmp equivalence relation $R'$ if there exists subsets of their definition domains  which intersect almost every class, say $X_0\subset X$ and $X_0'\subset X'$, and such that the restriction $R_{|X_0}$ and $R'_{|X_0'}$ are isomorphic. 
\end{definition}

Thus, if $Y$ is $R$-quasi-periodic, $Y'$ is $R'$-quasi-periodic, and $R$ is stably isomorphic to $R'$, then $Y$ and $Y'$ are considered to be quasi-periodic ``in the same way''. In fact, the random variable $Y$ can be viewed as a variable defined on the quotient space $\Omega=X/R$, namely, $Y\colon \Omega\to \Sigma$, and a stable equivalence class of equivalence relation can be viewed as a singular space structure on $\Omega$ (in the sense of \cite{Connes}). Following this point of view, the space $Y$ was said to be an $\Omega$-periodic space in \cite{qp} rather than a quasi-periodic $R$-space. 

The definition of a ``concept of quasi-periodicity'' as a stable isomorphism class of equivalence relation, is of course parallel to Mackey's definition of a ``virtual subgroup'' as a similarity class of groupoids \cite{Mackey}, and it was defined in this way in \cite{qp} by analogy. 

We shall only need the following lemma on this topic.

\begin{lemma}
Let $X_0\subset X$ be a subset intersecting every $R$-class, and  let $Y$ be an $R$-space. Then $Y$ is $R$-quasi-periodic if and only if $Y_{|X_0}$ is $R_0$-quasi-periodic, where $R_0:=R\cap X_0\times X_0$ denotes the restriction.  
\end{lemma}

\begin{proof} Since $X_0$ intersects every class, we may find a  partition $X=\bigsqcup_{k=0}^N X_k$ (where $N$ may be infinite) and a family $(\theta_k\colon X_k\to X_0)_{k=0}^N$  of partial inner automorphisms of $R$, where $\theta_0$ is the identity on $X_0$ and $\theta_k$ is, by definition, an embedding whose graph is included in $R$.

Suppose $Y$ is $R$-quasi-periodic and let $Y_0$ be a fundamental domain. For every $x\in X_0$ we let
\[
\widetilde Y_0^x=\bigcup_{k=0, \ \ x\in \ran \theta_k}^N (x,\theta_k^{-1}(x))Y_0^{\theta_k^{-1}(x)}.
\]
We claim that $\widetilde{Y_0}_{|X_0}$ is a fundamental domain for $Y_{|X_0}$. Let $x\in X_0$ and $y\in Y^x$. Since $Y_0$ is a fundamental domain, there exist $y_0\in Y_0$ and $r\in R$ such that $y=ry_0$. Let $k$ denote the unique integer such that $\s(r)\in X_k$. Let $x_0=\theta_k(s(r))\in X_0$ and let $\tilde y_0:=(x_0,s(r))y_0$. Then $(x,x_0)\in R_0$ and $\tilde y_0\in \widetilde Y_0^x$ since $\tilde y_0=(x_0,\theta^{-1}(x_0))y_0$.  Furthermore, $(x,x_0)\tilde y_0=(x,\theta^{-1}(x_0))y_0=ry_0=y$. Let us now prove uniqueness. Suppose $y,z\in \tilde Y_0$ verify $y=rz$ for some $r\in R_0$. Write $y=(x_0,\theta_k^{-1}x_0)y_0$ and $z=(x_0',\theta_l^{-1}x_0')z_0$. Then $(x_0,\theta_k^{-1}x_0)y_0=r(x_0',\theta_l^{-1}x_0')z_0$ and since $y_0,z_0\in Y_0$, we deduce: $(x_0,\theta_k^{-1}x_0)=r(x_0',\theta_l^{-1}x_0')$ and $y_0=z_0$. In particular, $\theta_k^{-1}x_0=\theta_l^{-1}x_0'$. Since the domains of $\theta_k$ and $\theta_l$ are disjoint for $k\neq l$, this implies $k=l$. Thus, $x_0=x_0'$. This shows that $y=z$ and $r=1$.

Suppose conversely that $Y_{|X_0}$ is quasi-periodic and let $Y_0$ be a fundamental domain.  We claim that $Y_0$ is a fundamental domain for $Y$. Let $x\in X_k$ and let $y\in Y^x$. Since $Y_0$ is a fundamental domain for $R_0$, there exists  $x_0\in X_0$ and $y_0\in Y_0$ such that $(x_0,\theta_k(x))\in R_0$ and $y_0=(x_0,\theta_k(x))(\theta_k(x),x)y$. Thus, $y=r^{-1}y_0$ where $r=(x_0,x)\in R$ and $y_0\in Y_0$. Let us prove uniqueness. Suppose $y=rz$ where $y,z\in Y_0$ and $r\in R$. Since $y,z\in Y_0$ we have $\s(r),\r(r)\in X_0$.   Therefore,  $r\in R_0$. Since $Y_0$ is a fundamental domain for $R_0$, this implies $r=1$ and therefore $y=z$.
\end{proof}

We will not develop quasi-periodicity further on this occasion and will conclude with a remark regarding Def.\ \ref{D - qp space}.  The fact that a quasi-periodic space is defined to be a field of spaces (or a random variable) rather than a single space in the category $\Sigma$ endowed with a ``quasi-periodic structure'', might seem unnatural a priori, but it is justified in concrete examples. Thus, an aperiodic tiling generates a quasi-periodic space in the sense above, by considering the closure of its translates in the local topology. In the case of Penrose's tilings, the kites and darts naturally generates uncountably many pairwise non isomorphic aperiodic tilings rather than a single space.  For tilings, $\Sigma$ can be chosen to be the category of  metric flat 2-complexes, $R$ is a topological equivalence relation, and a minimality assumption (meaning that every class is dense) ensures that every patch in any given realization, can be found in every other realization. In this case, $\Omega=X/R$ is  a ``topological concept of quasi-periodicity'' for the tiling rather than a ``measure-theoretic concept of quasi-periodicity'' that can be used, for instance, to study the concentration of measure property. In a measure theoretic context, an ergodicity assumption replaces the minimality assumption and will ensure that every patch can be found in almost every realization. In probabilistic terms, when $R$ is ergodic, the singular probability space $\Omega$ has a single ``generic point'' $\omega\in \Omega$, viewed as ``the'' random point in $\Omega$. and the quasi-periodic space $Y$ can be viewed as having a single ``generic'' realization $Y(\omega)$, called ``the'' random realization of $Y$. Any two such realizations are locally indistinguishable one from another.

\section{Treeable finite index subrelations} \label{S - treeable finite index}

 In this section we prove the following result. 

\begin{theorem}\label{T - finite index stable almost equality classes}
An equivalence relation that contains a treeable subrelation of finite index admits a quasi-periodic stable almost equality class over a quasi-periodic countable set.
\end{theorem}

In particular, if the relation  is itself treeable, then it admits a quasi-periodic stable almost equality class over a quasi-periodic set. The converse to this statement follows  by Prop.\ \ref{P - treeable qp} assuming that the answer to Question \ref{Q - ast} is positive.

Let $R$ be an equivalence relation. We follow the notation in \cite{FSZ}. A subrelation $S$ of $R$ induces each $x\in X$ an equivalence relation on the $R$-class $R(x)$. We let $J(x)=R(x)/S$ be the quotient space. If $R$ is ergodic, then the function $|J(x)|$ is essentially constant. Furthermore,  there exists functions $f_i\colon X\to X$ such that for every $x\in X$, $\{S(f_i(x)) : 1\leq i\leq n\}$ forms a partition of $R(x)$, where $n=|J(x)|$. Functions satisfying these conditions are called \emph{choice functions} for the pair $S\subset R$  (see \cite[\ts 1]{FSZ}). We  assume that $f_1=1$ is the identity on $X$. We say that $S$ is a subrelation of index $n$ in $R$ if it admits $n$ choice functions. Since the function $|J(x)|$ is invariant, every finite index subrelation of $R$, i.e., such that $|J(x)|\leq N<\infty$ almost surely, is a  sum of at most $N$ components of $R$ in restriction to which $S$ is an index $n$ subrelation for some $n\leq N$, and each component can be treated separately.

Suppose $S$ is a subrelation of index $n<\infty$ in $R$.  Let $S_n$ denote the symmetric group. Let $\sigma\colon R\to  S_n$ be defined by $\sigma(x,y)(i)=j$ if and only if $S(f_i(y))=S(f_j(x))$. Then $\sigma\in Z^1(R,S_n)$ and $[\sigma]\in H^1(R,S_n)$ is independent of the choice functions $f_i$. The map $\sigma$ is called the \emph{index cocycle} of the inclusion (\cite[\ts 1]{FSZ}).

We first prove that inducing a quasi-periodic $S$-set provides a quasi-periodic $R$-set.  

We view $R$ as a right $R$-set, and the choice functions as  a fundamental set of sections for the right action of $S$ on $R$; thus, $R$ is quasi-periodic as a right $S$-set and 
\begin{align*}
R^x&=\bigsqcup_{i=1}^n (x,f_i(x))S\\
&=\bigsqcup_{i=1}^n \{(x,y) : (y,f_i(x))\in S\}
\end{align*}
 for every $x\in X$, and for every $(x,z)\in R$, there exists a unique index $i=1,\ldots, n$ such that $(f_i(x),z)\in S$. 

Let $E$ be an $S$-set. For each  $x\in X$ we  define
\[
(R\otimes_S E)^x:=\bigsqcup_{i=1}^n (x,f_i(x))\otimes E^{f_i(x)}.
\]
These sets form a field $R\otimes_S E$ over $X$ which is naturally an $R$-set under the action:
\[
(x,y)(y,f_i(y))\otimes e:=(x,f_j(x))\otimes (f_j(x),f_i(y)) e
\]
for each $(x,y)\in R$, each $i=1,\ldots, n$,  each $e\in E^{f_i(y)}$, and $j=\sigma(x,y)(i)$ is the unique index such that $(f_i(y),f_j(x))\in S$. 

\begin{lemma}
If $E$ is a quasi-periodic countable $S$-set, then $R\otimes_S E$ is a quasi-periodic countable $R$-set.
\end{lemma}

\begin{proof}
The fact $R\otimes_S E$ is an $R$-set is a straightforward computation which follows by the definition of the 1-cocycle $\sigma$.

Let us prove quasi-periodicity. Let $F$ be a fundamental domain for the $S$-action on $E$. Let $g=(x,f_j(x))\otimes h\in (R\otimes_S E)^x$.  Since $F$ is a fundamental domain, there exists a unique $y\in X$ and a unique $f\in F^y$ such that $(f_j(x),y)\in S$ and $h=(f_j(x),y)f$.  Thus,  
\[
g=(x,f_j(x))\otimes h=(x,f_j(x))\otimes  (f_j(x),y)f=(x,y)( (y,y)\otimes f).
\] 
showing that $1\otimes F$ intersects every $R$-orbit. 

To establish uniqueness, assume that $f$ and $g$ are two elements of $1\otimes F$ such that $f=(x,y)g$. Consider the unique index $j$ such that $(f_j(x),y)\in S$ and write $g=(y,y)\otimes h$ for $h\in F^y$. Then
\[
(x,y)g=(x,f_j(x))\otimes (f_j(x),y)h.
\] 
Since $f\in 1\otimes F$, this implies $j=1$, and  $f_j=f_1=1$ implies $(x,y)\in S$. Setting $f=(x,x)\otimes e$ for $e\in F^x$ we obtain that $(x,y)g=(x,y)((y,y)\otimes h)=(x,x)\otimes (x,y)h$. Thus, $e=(x,y)f$. Since $e\in F^x$ and $f\in F^y$ and $F$ is a fundamental domain for $S$, this implies $x=y$, and therefore $f=g$. 

It is obvious that $1\otimes F$ is measurable since $F$ is. This proves that $1\otimes F$ is a fundamental domain for $R$.
\end{proof}

Next we show that induction preserves almost stability.

Suppose $V$ sits in an almost equality class in $\G(E,\ZI_2)$. For each $v\in V$, consider the element $R\otimes_S v$ in $\G(R\otimes_S E,\ZI_2)$ defined by
\[
(R\otimes_S v)^x((x,f_i(x))\otimes e):=v^{f_i(x)}(e),\ \ \ e\in E^{f_i(x)}, \ \ \ i=1,\ldots,n.
\]
It is clear that $R\otimes_S v$ is a Borel section; we shall write $R\otimes_S V$ for the set of all such functions. 

\begin{lemma}\label{L - induction stable ae class} 
The set $R\otimes_S V$ sits in an almost equality class in $\G(R\otimes_S E,\ZI_2)$.  If furthermore $V$ sits in a uniform almost equality class, then so does $R\otimes_S V$. If  $V$ sits in an $S$-stable almost equality class, then $R\otimes_S V$ sits in an $R$-stable almost equality class in $\G(R\otimes_S E,\ZI_2)$, and the following equality holds
\[
|(x,y)(R\otimes_S v)^y\bigtriangledown(R\otimes_S v)^x|=\sum_{i=1}^n |((f_i(x),f_{\sigma(x,y)(i)}(y)) (v\mid E^{f_{\sigma(x,y)(i)}(y)}) \bigtriangledown (v\mid E^{f_i(x)})|
\]
for each $(x,y)\in R$, where $\sigma$ is the index 1-cocycle. 
\end{lemma}

\begin{proof} We note first that $R\otimes_S V$ sits in an almost equality class.  If $u,v\in V$ and $x\in X$ then
\begin{align*}
&|(R\otimes_S v)^x\bigtriangledown(R\otimes_S u)^x|\\
 &\hskip2cm=\sum_{i=1}^n \sum_{e\in E^{f_i(x)}}(R\otimes_S v)^x((x,f_i(x))\otimes e)-(R\otimes_S u)^x ((x,f_i(x))\otimes e))\\
&\hskip2cm=\sum_{i=1}^n |(v\mid E^{f_i(x)})\bigtriangledown (u\mid E^{f_i(x)})|<\infty.
\end{align*}
 If furthermore $V$ sits in a uniform almost equality class, then so does $R\otimes_S V$, and we have an inequality
 \[
 \|R\otimes_S v\td R\otimes_S u\|_\infty \leq n \|u\td v\|_\infty
 \]
 by the above estimate.

Suppose $V$ is almost $S$-stable and let us check almost $R$-stability.
Let $v\in V$ and let $(x,y)\in R$. For each $i=1,\ldots, n$ we write $j=\sigma(x,y)(i)$. Thus, $j$ is an index which depends on $i$ and the unique index $j= 1,\ldots, n$ such that $(f_i(x),f_j(y))\in S$. We have
\begin{align*}
&|(x,y)(R\otimes_S v)^y\bigtriangledown(R\otimes_S v)^x|\\
&\hskip1cm=\sum_{i=1}^n \sum_{e\in E^{f_i(x)}}((x,y)(R\otimes_S v)^y)((x,f_i(x))\otimes e)-(R\otimes_S v)^x((x,f_i(x))\otimes e))\\
&\hskip1cm=\sum_{i=1}^n \sum_{e\in E^{f_i(x)}}(R\otimes_S v)^y((y,x)((x,f_i(x))\otimes e)-(R\otimes_S v)^x((x,f_i(x))\otimes e))\\
&\hskip1cm=\sum_{i=1}^n \sum_{e\in E^{f_i(x)}}(R\otimes_S v)^y((y,x)((x,f_i(x))\otimes e)-(R\otimes_S v)^x((x,f_i(x))\otimes e))\\
&\hskip1cm=\sum_{i=1}^n \sum_{e\in E^{f_i(x)}}(R\otimes_S v)^y((y,f_j(y))(f_j(y),f_i(x))\otimes e)-(R\otimes_S v)^x((x,f_i(x))\otimes e))\\
&\hskip1cm=\sum_{i=1}^n \sum_{e\in E^{f_i(x)}}v^{f_j(y)}((f_j(y),f_i(x)) e)-v^{f_i(x)}(e)\\
&\hskip1cm=\sum_{i=1}^n |((f_i(x),f_j(y)) (v\mid E^{f_j(y)}) \bigtriangledown (v\mid E^{f_i(x)})|
\end{align*}
which is finite because $v$ sits in an $S$-stable almost equality class.
\end{proof}

Suppose $V$ is almost $S$-stable. Then $R\otimes_S V$ is almost $R$-stable and the restriction map 
\begin{align*}
\rho\colon R\otimes_S V&\to \G(1\otimes E,\ZI_2)\\
R\otimes_S v&\mapsto (R\otimes_S v)_{|1\otimes E}
\end{align*} 
can be viewed as a map from $R\otimes_S V$ to $\G(E,\ZI_2)$ under the identification $1\otimes E\simeq  E$ taking $1\otimes e$ to $e$. Furthermore,  since $f_1=1$ we have 
\[
(R\otimes_S v)^x(1\otimes e)=v^x(e)
\]
for any $v\in V$ and any $e\in E^x$.
Thus, we may assume that $\rho$ is in fact the map of $R\otimes_S V$ into $V$ taking  $R\otimes_S v$ to $v$. 

Let $\widetilde{R\otimes_S V}$ denote the stable almost equality class in which $R\otimes_S V$ sits. Being the restriction map, $\rho$ extend to a map (again denoted $\rho$) of $\widetilde{R\otimes_S V}$ into the almost equality class  $\widetilde V$ of $V$.

\begin{lemma}\label{L - f is S equivariant}
The map $\rho\colon \widetilde{R\otimes_S V}\to \widetilde V$ is an $S$-map. In particular, if $\widetilde V$ is a quasi-periodic $S$-space, then $\widetilde{R\otimes_S V}$ is a quasi-periodic $R$-space.
\end{lemma}

\begin{proof}
Let $(x,y)\in S$, $w\colon R\otimes_SE\to \ZI_2$ be almost equal to an element of $R\otimes_S V$, and $1\otimes e\in (1\otimes E)^x$. Since $f_1=1$ we have
\[
(y,x)(x,f_1(x))=(y,f_1(y))(y,x)=(y,x)\in S
\]
and therefore:
\begin{align*}
\rho((x,y)w^y)(1\otimes e)=w^y((y,x)(1\otimes e))=((x,y)w^y)(1\otimes e)=((x,y)\rho(w)^y)(1\otimes e).
\end{align*}
This proves that $\rho$ is an $S$-map.

If furthermore $\widetilde V$ is quasi-periodic then, since $\rho$ is an $S$-map, it follows by Lemma \ref{L - qp  R  map } that the space $ \widetilde{R\otimes_S V}$ is a quasi-periodic $S$-space. Since $S$ has finite index in $R$, it follows by Lemma \ref{L - qp finite index} that $ \widetilde{R\otimes_S V}$ is a quasi-periodic $R$-space.
\end{proof}

Suppose now that $S$ is a treeable equivalence relation.  Let $K_S\subset S$ be a treeing of $S$ and let $K_S^{*p}$ denote the set of non-backtracking paths of length $p$. If $S$ has infinite classes, then the natural projection $K_S^{*p}\to X$ onto the first coordinate is essentially surjective; otherwise, we adjoin to $K_S^p$ the element $(x,x)$ for any $x$ not in the image. Let us fix for every $p\geq 1$ a  countable set $L_p$ consisting of sections of $K_S^{*p}$ whose union of graphs coincide with $K_S^{*p}$. We also let $L_0:=\{1\}$ where $1$ is the identity map of $X$. Thus, if $k\in L_p$ then by definition, $(x,k(x))$ is a vertex of the Cayley tree $T$ of $S$ relative to $K_S$ at distance at most $p$ from $(x,x)$.  For every $p\geq 0$ and every $k\in L_n$ we let $v_{p,k}(x)$ denote the set of edges of $T^x$ which point to $(x,k(x))$. Let $V=\{v_{p,k} : p\geq 1, k\in L_n\}$. This is a subset of $\G(ET,\ZI_2)$.

\begin{lemma} \label{L - induction ae class} 
Suppose that $S$ is treeable of  index $n<\infty$ in  $R$. 
 Then $R\otimes_S V$ sits in a uniform $R$-stable  almost equality class. \end{lemma}

\begin{proof} We note first  that $V$ sits in a uniform almost equality class. Indeed, if $v_{p,k}$ and $v_{q,l}$ are two elements in $V$, and $d^x$ refers to the distance at $x$ in the Cayley tree of $S$ relative to $K_S$, then 
\[
|v_{p,k}(x)\td v_{0,1}(x)| \leq d^x((x,kx),(x,x))\leq p \text{ and } |v_{q,l}\td v_{0,1}|\leq d^x((x,lx),(x,x))\leq q.
\] 
Thus, $\|v_{p,k}\td v_{q,l}\|_\infty\leq p+q$.  

 It follows by Lemma \ref{L - induction stable ae class} that $R\otimes_S V$ sits in a uniform  almost equality class.  Namely, $|(R\otimes_S v_{p,k})^x\bigtriangledown(R\otimes_S v_{q,l})^x|$ is bounded by:
\begin{align*}
|(R\otimes_S v_{p,k})^x\bigtriangledown(R\otimes_S v_{q,l})^x|&=\sum_{i=1}^n \sum_{e\in E^{f_i(x)}}(R\otimes_S v_{p,k}((x,f_i(x))e)-R\otimes_S v_{q,l}((x,f_i(x))e))\\
&=\sum_{i=1}^n |(v_{p,k}\mid E^{f_i(x)})\bigtriangledown (v_{q,l}\mid E^{f_i(x)})|\\
&=n(p+q)<\infty.
\end{align*}
Furthermore, since $V$ sits in an $S$-stable almost equality class, $R\otimes_S V$ sits in an $R$-stable almost equality class.  
\end{proof}

\begin{proof}[Proof of Theorem \ref{T - finite index stable almost equality classes}]
Lemma \ref{L - induction stable ae class} shows that $R\otimes_S V$ sits in a stable almost equality class $ \widetilde{R\otimes_S V}$ in $(R\otimes_S ET,\ZI_2)$. It remains  to prove that $ \widetilde{R\otimes_S V}$ is a quasi-periodic $R$-set. We prove that  $\widetilde{V}$ is a quasi-periodic $S$-space. The result will then follow by Lemma \ref{L - f is S equivariant}. 

 Let $\tilde v\in \G(\widetilde V)$ sit in the uniform almost equality class of $V$. Take $v\in V$ and a uniformly bounded element $a\in \G(ET,\ZI_2)$ such that $\tilde v= v+a$. For every $(x,y)\in S_{\tilde v}$ we have
\[
(x,y)v^y-v^x=(x,y)a^y-a^x.
\]
Let $N$ be an upper bound on $|a^x|$. Then, 
\[
|(x,y)v^y-v^x|\leq 2N.
\]
This shows that the distance between $v^x$ and $(x,y)v^y$, viewed as elements of $VT^x$, is at most $2N$ for every $(x,y)\in S_{\tilde v}$. Consider the ball $B^x$ of smallest radius in the tree $T^x$ which contains the set 
\[
\{(x,y)v^y : (x,y)\in S_{\tilde v}\}.
\]
Clearly, $B$ is a field of subsets of $VT$ which is invariant under $S_{\tilde v}$. Since $S_{\tilde v}$ acts isometrically, the barycenter of $B$ forms a field of points which is invariant under $S_{\tilde v}$. These points might be vertices or belong to edges in $T^x$. In both cases, we may find a section $\tilde v_0\in \G(VT)$ which is $S_{\tilde v}$-invariant. 

Consider now the map
\begin{align*}
f_{\tilde v} \colon S\tilde v &\to S v_0\\ 
(x,y)\tilde v^y&\mapsto (x,y)v_0^y.
\end{align*}
Note that this map is a well-defined Borel map since if $(x,y)\tilde v^y=(x,z)\tilde v^z$ then $(y,z)\in S_{\tilde v}$ so $(x,y)\tilde v_0^y=(x,z)\tilde v_0^z$. Furthermore, $f_{\tilde v}$ is an $S$-map by definition. Since $VT$ is $S$-quasi-periodic, so is $S\tilde v$.  Lemma \ref{L - qp sections} implies that $\tilde V$ is an $S$-quasi-periodic space. 
\end{proof}

\end{document}